\documentclass[11pt]{article} 
\usepackage[utf8]{inputenc}
\usepackage{amsmath,amsthm,amssymb,amsfonts}
\usepackage{url}
\usepackage{color}
\usepackage{caption,subcaption}
\usepackage{textcomp}
\usepackage{rotating}
\usepackage{siunitx}
\usepackage{multirow}
\usepackage{float}
\usepackage{pdflscape}
\usepackage{afterpage}
\usepackage{cleveref}
\usepackage{soul}
\usepackage[linesnumbered, ruled]{algorithm2e}
\usepackage{pseudocode}
\usepackage[toc,page,titletoc,title]{appendix}

\newtheorem{theorem}{Theorem}[section]
\theoremstyle{plain}
	
\newtheorem{definition}[theorem]{Definition}
\newtheorem{construction}[theorem]{Construction}

\newtheorem{proposition}[theorem]{Proposition}

\newtheorem{observation}[theorem]{Observation}
\newtheorem{remark}[theorem]{Remark}
\newtheorem{question}[theorem]{Question}

\author{Gerold J\"ager \\ gerold.jager@umu.se 
	\and Klas Markstr\"om \\ klas.markstrom@umu.se 
	\and Denys Shcherbak \\ denys.shcherbak@umu.se
	\and Lars-Daniel \"Ohman \\ lars-daniel.ohman@umu.se 
	}

\date{}

\begin{document}

\title{Enumeration of Row-Column Designs}

\maketitle

\begin{abstract}
	We computationally completely enumerate a number of types of 
	row-column designs up to isotopism, including double, sesqui and 
	triple arrays as known from the literature, and two newly
	introduced types that we call mono arrays and AO-arrays.
	We calculate autotopism group sizes for the designs we generate.

	For larger parameter values, where complete enumeration is not
	feasible, we generate examples of some of the designs, and generate 
	exhaustive lists of admissible parameters. For some admissible 
	parameter sets, we prove non-existence results. We also give some 
	explicit constructions of sesqui arrays, mono arrays and AO-arrays, 
	and investigate connections to Youden rectangles and binary pseudo 
	Youden designs.	
\end{abstract}

\section{Introduction}
\label{sec_intro}

A classic object of study both in combinatorial design theory
and practical experimental design is \emph{balanced incomplete block
designs} (BIBD:s). In such designs, blocks are unordered sets, but 
in the experimental design setting, \emph{order}
among elements in blocks is sometimes important, and designs with
different types of ordering within blocks have likewise been studied,
for example \emph{Youden rectangles} introduced by Youden in~\cite{You}. 
The interplay between ordered and unordered designs is sometimes rather simple.
For example, the elements in the blocks of any \emph{symmetric} BIBD 
(SBIBD) can be ordered to become the columns of a Youden rectangle, 
and, vice versa, forgetting the ordering in the columns of a Youden
rectangle yields an SBIBD.

A less well understood connection is between a type of designs introduced 
by Agrawal~\cite{agrawal}, now known as \emph{triple arrays} on the one
hand, and SBIBD:s on the other hand, where Agrawal indirectly conjectured
that any SBIBD can be used to construct a triple array. The converse, that
existence of a triple array implies the existence of an SBIBD for corresponding
parameters, was proven by McSorley et al.~\cite{McSorley}. SBIBD:s have 
been studied extensively, but the literature on related ordered objects
is more scarce.

In the present paper we therefore, for small parameters, completely enumerate 
several classes of \emph{row-column designs}, that is, designs where blocks 
are interpreted as columns, and elements in blocks are ordered into rows. 
The term \emph{two-way designs} has also been used in the literature
for this class of designs, which includes triple arrays, but also other 
ordered designs, which we present below. In addition to the complete 
enumeration, we also 
provide examples of such designs produced heuristically for larger 
parameters, prove non-existence for a few parameter sets, and give some 
general constructions of families of row-column designs.

Our interest in these designs is mainly combinatorial, but they are also
interesting from an applied experimental design point of view. The set of 
small examples of designs we produce are readily applicable as experimental
designs, but may also prove useful as inspiration for further general 
constructions of families of designs.

The rest of the paper is organized as follows. In Section~\ref{sec_not}, 
we give basic definitions and describe the notation used.
In Section~\ref{sec_parameters} we investigate what parameters
are admissible for these designs and prove a non-existence result, 
followed by Section~\ref{sec_SAconstr} where we prove some results on 
constructing sesqui- mono- and AO-arrays.
In Section~\ref{sec_gen}, we describe briefly the computational methods 
employed. Section~\ref{sec_compresults} contains the main results from 
our computations, and Section~\ref{sec_other_rel} investigates connections
to Youden rectangles and binary pseudo Youden designs. 
Section~\ref{sec_concl} concludes with two open questions.

\section{Basic notation and definitions}
\label{sec_not}

\subsection{Row-column designs}

A \emph{row-column design} is an $r \times c$
array on $v$ symbols, where no symbol is used more than once
in any row or any column, i.e., the design is \emph{binary}. We shall 
also assume throughout that the array is \emph{equireplicate} with 
replication number $e$, that is, each symbol appears exactly $e$ 
times in the array.

Balanced incomplete block designs (BIBDS:s) are a staple of combinatorial
design theory, so we only briefly mention that we will infringe
on the standard names of parameters, $(v,b,r,k,\lambda)$-BIBD (number 
of symbols, number of blocks, number of blocks containing a given symbol, 
number of symbols in a block, number of blocks containing any 2 
distinct symbols), or $(v,k,\lambda)$-BIBD for short, by using $r$ not 
for the number of blocks, but for the number of rows in our row-column 
designs. The symbol $v$ has the corresponding role in both BIBD:s and
row-column designs.

Looking at the symbol sets in the single rows, or the single columns, 
a row-column design can satisfy additional intersection properties:

\begin{itemize} 
	\item[(RR)]
		Any pair of distinct rows contains $\lambda_{rr}$ common symbols.
	\item[(CC)]
		Any pair of distinct columns contains $\lambda_{cc}$ common symbols.
	\item[(RC)]
		Any pair of one row and one column contains $\lambda_{rc}$ common symbols.
\end{itemize}

The property RC is usually referred to as \emph{adjusted orthogonality}. 
The terminology for the row-column designs depends on which of these 
properties are satisfied. Treating these possibilities in a structured
fashion, we shall begin with row-column designs satisfying all the 
properties, and then successively relax one or more of the 
properties.

\begin{definition}
	A $(v, e, \lambda_{rr}, \lambda_{cc}, \lambda_{rc} : r \times c)$ 
	\emph{triple array} is an $r \times c$ row-column design on $v$ symbols 
	with replication number $e$, satisfying properties RR, CC and RC.
\end{definition}

Triple arrays were introduced by Agrawal~\cite{agrawal}, though examples
were known earlier. A good general introduction to the study 
of triple arrays is given by McSorley et al. in~\cite{McSorley}. 
Clearly, the transpose of an 
$r \times c$ triple array is a $c \times r$ triple array, with
parameters $\lambda_{rr}$ and $\lambda_{cc}$ interchanged.

\begin{definition}
	A $(v, e, \lambda_{rr}, \lambda_{cc}, - : r \times c)$ 
	\emph{double array} is an $r \times c$ row-column design on $v$ symbols 
	with replication number $e$, satisfying properties RR and CC. If
	property RC is expressly forbidden to hold, the double array is called
	a \emph{proper} double array.
\end{definition}

Double arrays have generally been studied alongside of triple arrays, 
for example in~\cite{McSorley}, and any triple array is \emph{a fortiori}
also a double array (but of course not a \emph{proper} double array). As 
for triple arrays, the transpose of a double array is again a double 
array.

\begin{definition}
	A $(v, e, \lambda_{rr}, -, \lambda_{rc}: r \times c)$ 
	\emph{sesqui array} is an $r \times c$ row-column design on $v$ symbols 
	with replication number $e$, satisfying properties RR and RC. If
	property CC is expressly forbidden to hold, the sesqui array is called
	a \emph{proper} sesqui array.
\end{definition}

Sesqui arrays were introduced with this name in Bailey, Cameron and 
Nilson~\cite{sesqui}, though
examples had appeared earlier, for instance in Bagchi~\cite{Bag98}.
With this definition, the transpose of a sesqui array is only then a 
sesqui array when it is a triple array. It seems unnecessary to introduce 
new terminology for a row-column design satisfying properties CC and RC, 
so we shall simply call such designs \emph{transposed sesqui arrays}.
A sesqui array with parameters $(v, e, \lambda_{rr}, -, \lambda_{rc}: r \times c)$
is then equivalent to a transposed sesqui array with parameters
$(v, e, -, \lambda_{rr}, \lambda_{rc}: c \times r)$.
Because of how we generate designs computationally, column by column, 
we mainly work with transposed sesqui arrays.
Any triple array is clearly also a (transposed) sesqui array, but not a \emph{proper} 
(transposed) sesqui array.

\begin{definition}
	A $(v, e, -, \lambda_{cc}, - : r \times c)$ 
	\emph{mono array} is an $r \times c$ row-column design on $v$ 
	symbols with replication number $e$, satisfying property CC. If
	properties RR and RC are expressly forbidden to hold, the mono 
	array is called \emph{proper}.
\end{definition}

Mono arrays have to our knowledge not been studied before. We 
include them here in order to be able to study the relation between 
the intersection conditions, and note that as for sesqui arrays, it 
makes sense to talk about \emph{transposed mono arrays}. We are mainly
concerned with the non-transposed variant. Also, 
any triple, double, or transposed sesqui array is also a mono array, 
but not a \emph{proper} mono array.

\begin{definition}
	A $(v, e, -, -, \lambda_{rc} : r \times c)$ 
	\emph{adjusted orthogonal array} (AO-array for short) is an 
	$r \times c$ row-column design on $v$ symbols with replication 
	number $e$, satisfying property RC. If properties RR and CC are 
	expressly forbidden to hold, the AO-array is called \emph{proper}.
\end{definition}

Transposes of AO-arrays are again AO-arrays. AO-arrays 
(our term) have been studied by e.g. Bagchi~\cite{Bag96}, under
a range of slightly different names, sometimes without the assumption
of equireplication.

\subsection{Notions of equivalence}

There are several possible notions of equivalence between the objects 
studied in the present paper. We shall consider two row-column designs 
$A$ and $B$ to be equivalent if a permutation of the rows, a permutation 
of the columns and a permutation of the symbols in the design $A$
can yield the design $B$. This notion of equivalence is usually referred 
to as \emph{isotopism}, so that $A$ and $B$ would be called isotopic.
An isotopism that carries $A$ to itself is called an \emph{autotopism},
and the set of autotopisms of a row-column design forms the \emph{autotopism
group}, in a standard way.

Isotopism is the most natural notion of equivalence in this setting,
since the number of rows, the number of columns, and the number of 
symbols are generally all distinct, and therefore taking transposes or 
interchanging the roles of, e.g., symbols and rows is only possible 
in certain special cases. In our data, the only interesting case where 
some of these parameters coincide is when $r=c$. In these few instances, 
we have also investigated the number of equivalence classes, called
\emph{trisotopism classes}, we get when also allowing transposes. 
For a more in-depth discussion of different 
relevant notions of equivalence, see Egan and Wanless~\cite{WanlessEgan}.

\section{Parameter sets}
\label{sec_parameters}

To avoid trivial examples, we require that $r, c > 1$. 
Regarding $v$, if $v < \max\{r,c\}$, then the arrays can't be
binary, but we also exclude $v = \max\{r,c\}$ to avoid including
Latin rectangles and Youden rectangles, that have been enumerated 
by McKay and Wanless~\cite{WanlessMcKay} and the current 
authors~\cite{YoudenEnum}, respectively.

Similarly, if $v>rc$, then the arrays can't be equireplicate, but we 
also exclude the case $v=rc$, where the replication number 
would be $e=1$, since such designs trivially exist, are unique up to 
isotopism, and are less interesting from a design point of view. The next
possible integer value is $e=2$, where we get $v = \frac{rc}{2}$.
The range of values we are interested in for $v$ is therefore
$\max\{r,c\} < v \leq \frac{rc}{2}$.

\subsection{Admissible parameters}

Even apart from the initial observations and assumptions above, the parameters 
involved in row-column designs are not independent of each other. For 
example, equireplication immediately implies that $ev = rc$, by a standard 
double count. Depending on what further intersection properties we want 
to hold, there will be additional necessary divisibility conditions on the 
other parameters. Note that binarity does not impose any additional 
divisibility constraints.

Bagchi~\cite{Bag98} observed that in any binary, equireplicate row-column
design with adjusted orthogonality (property RC), the constant size 
$\lambda_{rc}$ of the row-column intersections equals the replication 
number $e$. Therefore $\lambda_{rc} = e = \frac{rc}{v}$ holds for triple arrays, 
(proper, transposed) sesqui arrays and (proper) AO-arrays. 
Similarly, if row-row intersections 
have constant size, then $\lambda_{rr}=\frac{c(e-1)}{(r-1)}$, and 
for constant size column-column intersections 
$\lambda_{cc}=\frac{r(e-1)}{(c-1)}$ (see McSorley et al. \cite{McSorley}).

Since all these parameters must be integers, this forces some rather 
restrictive divisibility conditions. As in the Handbook of Combinatorial
Designs~\cite{handbook} for parameter sets for BIBD's, we use the 
term \emph{admissible} for a set of parameters in this context (for a 
certain type of design), if the parameter set can't be ruled out on 
these basic divisibility conditions alone.

\subsection{Average intersection sizes and non-existent proper arrays}
\label{sec_average}

The condition on $\lambda_{rc}$ observed by Bagchi can be generalized 
slightly to designs without adjusted orthogonality. 
Let $\overline{\lambda_{rc}}$ be the average size of the 
row-column intersections in a row-column design.

\begin{proposition}\label{double_average}
	In any equireplicate binary row-column design, 
	$\overline{\lambda_{rc}} = e = \frac{rc}{v}$.
\end{proposition}

\begin{proof}
	Let $I_{ij}$ denote the intersection between row $i$ and column $j$,
	and let $x$ be any symbol in the array. Since there are no repeats in
	any row or column, $x$ appears in $e$ columns, and in $e$ rows, and
	so it appears in $e^2$ of the sets $I_{ij}$. The total number of 
	symbols is $v$, so the sum of all the cardinalities of the $I_{ij}$
	is $ve^2$. Since there are $r$ rows and $c$ columns, there are $rc$ 
	sets $I_{ij}$, and thus $ve^2 = rc \overline{\lambda_{rc}}$. Since
	$ev = rc$, we get $\overline{\lambda_{rc}} = e$.
\end{proof}

Clearly, if $\overline{\lambda_{rc}} = \frac{rc}{v}$ is an integer, one 
way of attaining this as the average size of the row-column intersections 
is by all row-column intersections having constant size 
$\overline{\lambda_{rc}}$. A set of parameters is therefore admissible 
for exactly one of the following sets of design types:
\begin{itemize}
	\item Only AO-arrays
	\item AO-arrays, mono arrays, and transposed sesqui arrays
	\item AO-arrays, transposed mono arrays, and sesqui arrays
	\item All designs considered here
\end{itemize}

Similarly as for $\overline{\lambda_{rc}}$, when $v$ (number of symbols), 
$r$ and $c$ (number of rows and columns) have been set in a binary
equireplicate array, using double counting arguments, the averages 
$\overline{\lambda_{rr}}$ and
$\overline{\lambda_{cc}}$ can be calculated by 
$\overline{\lambda_{rr}}=\frac{c(e-1)}{(r-1)}$, and 
$\overline{\lambda_{cc}}=\frac{r(e-1)}{(c-1)}$. Using these averages, 
the possible sizes of intersections can exclude 
the possibility of some \emph{proper} arrays.

Note that if 
$2r-v > \overline{\lambda_{rr}}$ or $2c-v > \overline{\lambda_{cc}}$, 
then the parameter sets are not admissible.

\begin{proposition}\label{prop_full}
	Let $A$ be a binary, equireplicate $r \times c$ array on $v$ symbols.
	If $2c-v = \overline{\lambda_{rr}}=\frac{c(e-1)}{(r-1)}$,
	then the property RR holds.
	If $2r-v = \overline{\lambda_{cc}}=\frac{r(e-1)}{(c-1)}$,
	then the property CC holds.
\end{proposition}

\begin{proof}	
	With $v$ symbols, any two rows of length $c$ share at least 
	$2c-v$ symbols. If this equals $\overline{\lambda_{rr}}$, all 
	pairs of rows share exactly	$\overline{\lambda_{rr}}$ symbols, 
	so the row intersection property RR holds. A similar argument
	holds for the column intersection property CC.
\end{proof}

In the range of parameters we investigate,
Proposition~\ref{prop_full} may be used to rule out the existence
of proper AO-arrays for parameter sets $(v,r,c)$  in
\[
        \{ (6,4,3), (8,6,4), (9,6,3), (10,8,5), (12,8,3), (12,9,4),
        (12,10,6), (14,12,7)\}
\]
and the corresponding transposes. We may also rule out
proper mono arrays and proper transposed sesqui arrays for
parameter sets $(v,r,c)$ in
\[
        \{ (6,3,4), (12,4,9)\}.
\]

\begin{remark}\label{rem_full}
	Assuming that $1<r \leq e$, the condition 
	$2c-v = \overline{\lambda_{rr}}=\frac{c(e-1)}{(r-1)}$ in 
	Proposition~\ref{prop_full} can be rewritten as $r=e+1$, by substituting
	$v=\frac{rc}{e}$ and simplifying.
	Correspondingly, the condition 
	$2r-v = \overline{\lambda_{cc}}=\frac{r(e-1)}{(c-1)}$ can be rewritten
	as $c=e+1$ if $c\neq1$ and $c\leq e$.
\end{remark}

\subsection{Component designs and non-existence}

If a row-column design $D$ has constant row-row intersection sizes, we may 
define a balanced incomplete block design, the \emph{row design}, which we 
denote by $BIBD_R(D)$ or simply $BIBD_R$ if it is clear from the context 
what $D$ is, in the following way. 

Label the rows of $D$ by 
$R_1, R_2, \ldots, R_r$, and let them be the $r$ symbols of $BIBD_R$. For
each symbol $x$ in $D$, define a block $B_x$ in $BIBD_R$ by including $R_i$ 
in $B_x$ iff $x$ appears in row $R_i$ in $D$. All blocks in $BIBD_R$  
have the same size $e$ since $D$ is binary and equireplicate, and each symbol 
$R_i$ appears in $c$ blocks, because each row contains $c$ unique symbols.
Each pair of symbols $R_i$ and $R_j$ in $BIBD_R$ appear together in 
$\lambda_{rr}$ blocks, since any such pair of rows intersect in $\lambda_{rr}$
symbols in $D$.

In standard parameter order for BIBD:s, the $BIBD_R$ has parameters 
$(r,v,c,e,\lambda_{rr})$, or shorter $(r,e,\lambda_{rr})$ (number of
symbols, number of symbols in a block, number of blocks containing any 2
distinct symbols). Similarly, 
if column-column intersection sizes are constant, we may define the 
\emph{column design} $BIBD_C$ which has parameters 
$(c,v,r,e,\lambda_{cc})$, or $(c,e,\lambda_{cc})$ for short.

Some further sets of parameters can be ruled out for mono, 
double, sesqui, and triple arrays, if the required row design or column
design is known not to exist. For example, as pointed out by McSorley et al. 
in~\cite{McSorley},
there is no $(21,7,15)$ double array, since the corresponding 
column design $BIBD_C$, with parameters $(15,21,7,5,2)$, or $(15,5,2)$ for short, 
does not exist. It follows by the same argument that there is no
transposed sesqui array on parameters $(21,7,15)$, and similarly 
no transposed sesqui arrays on parameter sets $(21,14,15)$, $(28,8,21)$, and 
$(28,20,21)$. 

Even if the basic divisibility conditions are met,
and the required component BIBD:s exist, this is not sufficient in 
general for existence of row-column designs. 
The smallest example of this is the parameter set 
$(6,3,4)$ which satisfies all the conditions mentioned above. 
On these parameters, however, there is no triple array, but there are 
double arrays, as observed in~\cite{McSorley}, and additionally, there are 
sesqui arrays and mono arrays.

\section{Constructions of sesqui arrays and AO-arrays}
\label{sec_SAconstr}

In the literature, there are by now several constructions of triple arrays,
but we will restrict our discussion here to sesqui arrays and AO-arrays.
Bailey, Cameron and Nilson~\cite{sesqui} gave two general constructions of
families of sesqui arrays. One of these constructions is for
$(n+1)\times n^2$ sesqui arrays on $n(n+1)$ symbols, based on two Latin 
squares of orders $n\times n$ and $(n+1) \times (n+1)$, and works for any 
$n\geq2$. 
The other construction starts with a \emph{biplane} (that is, a BIBD with 
the same number of blocks as symbols, $v=b$, and $\lambda = 2$) and a selected
block $B$ in the biplane. It gives, except for $k=3$, a $k\times (v-k)$ 
sesqui array on $k(k-1)/2$ symbols, given that a condition on the 
intersections of $B$ with the other blocks is satisfied.

We now give a general construction for $r \times mc$ sesqui arrays, whenever
there exists an $r \times c$ sesqui array, which does not have to be proper.

\begin{construction}\label{constr_product}
	Let $S$ be a $(v, e, \lambda_{rr}, -, \lambda_{rc}: r \times c)$ sesqui array.
	Introduce $v$ sets $S_1, S_2, \ldots, S_{v}$ of distinct symbols, 
	such that $|S_i| = m$ for all $i$.
	
	Replace each of the $v$ symbols $s_1, s_2, \ldots s_{v}$ in $S$ with a 
	row of length $m$, and fill all of the rows corresponding to symbol $s_i$ 
	with the symbols from $S_i$ in arbitrary order to get an $r \times mc$ 
	array $A$.
\end{construction}

An example of Construction~\ref{constr_product} where $m=3$ is given in 
Figure~\ref{fig_constr}. The ``canonical'' way of entering the symbols 
from $S_i$ is to use 
the same order for each occurrence of $S_i$, but using different orders 
for the different occurrences of $s_i$ also works. Note that the 
column-column intersections will differ in size, and that the differences 
will in some sense be maximized when using the canonical ordering for 
each $S_i$. When using the canonical order, the column design will
be \emph{disconnected} (that is, the column/symbol incidence structure
is disconnected), but it is easy to produce arrays with connected 
column designs, as in Figure~\ref{connect}. The connectedness of the
row design follows from the the array being a sesqui array.

\begin{figure}[h]
\begin{subfigure}{0.90\linewidth}
\begin{center}
	\begin{tabular}{|r|r|r|}
		\hline
		3	&0	&1  \\
		2	&3	&0	\\
		1	&2	&3	\\
		0	&1	&2	\\
		\hline
	\end{tabular}
\subcaption{A $4 \times 3$ sesqui array $S$}
\end{center}
\end{subfigure}

\begin{subfigure}{0.90\linewidth}
\begin{center}
	\begin{tabular}{|rrr|rrr|rrr|}
		\hline
		J & K & L & A & B & C & D & E & F \\
		G & H & I & J & K & L & A & B & C \\
		D & E & F & G & H & I & J & K & L \\
		A & B & C & D & E & F & G & H & I \\
		\hline
	\end{tabular}
\subcaption{The canonical $4 \times 9$ sesqui array received from $S$ by 
	Construction~\ref{constr_product} for $m=3$.}
\end{center}
\end{subfigure}

\begin{subfigure}{0.90\linewidth}
\begin{center}
	\begin{tabular}{|rrr|rrr|rrr|}
		\hline
		J & K & L & B & C & A & D & E & F \\
		G & H & I & J & K & L & C & A & B \\
		D & E & F & I & G & H & J & K & L \\
		A & B & C & D & E & F & H & I & G \\
		\hline
	\end{tabular}
\subcaption{An alternative $4 \times 9$ sesqui array received from $S$ by 
	Construction~\ref{constr_product}, with connected column design.}\label{connect}
\end{center}
\end{subfigure}
\caption{Applying Construction~\ref{constr_product}.}\label{fig_constr}
\end{figure}

As can be seen in Figure~\ref{fig_constr}, a trivial sesqui array (that is,
with $v = r$) can yield a non-trivial sesqui array. For example, the 
construction can be used to produce an $n \times (m(n-1))$ non-trivial sesqui 
array for any $m$ and $n$, by starting with an $n \times (n-1)$ Latin 
rectangle on $n$ symbols, since all such Latin rectangles are also
$(n, n-1, \lambda_{rr} = n-2, -, \lambda_{rc}=n-1: n \times (n-1))$ sesqui 
arrays.

\begin{theorem}\label{thm_constr}
	The array $A$ in Construction~\ref{constr_product} 
	is an $r \times mc$ sesqui array with row-row intersection size 
	$\lambda_{rr}^A = m\lambda_{rr}$, row-column intersection 
	size $\lambda_{rc}^A = \lambda_{rc}$ and replication number $e$.
	If $m \geq 2$, the resulting sesqui array is proper.
\end{theorem}

\begin{proof}
	It is clear that the construction can be carried out in the way described, 
	given the sesqui array $S$.
	
	No symbol in $A$ is repeated in any 
	row, since the newly introduced symbols from any $S_i$ occur only at most
	once in any row, since $s_i$ is only used at most once in any row in 
	$S$. No symbol in $A$ is repeated in any column, since $S_i$ and $S_j$ are 
	disjoint if $i \neq j$, and no $s_i$ is repeated in any column in $S$.
	
	To see that $A$ is equireplicate
	with replication number $e$, note that $S$ is equireplicate with 
	replication number $e$, by definition, and that 
	each new symbol $\sigma_{ij} \in S_i$ occurs exactly once for each 
	occurrence of $s_i$ in $S$.
	
	To see that the row-row intersection size in $A$ is constant 
	$\lambda_{rr}^A = m\lambda_{rr}$,
	note that the symbols common to two rows in $A$ are the ones 
	corresponding to the $\lambda_{rr}$ symbols common to two rows in 
	$S$, and each such symbol in $S$ is replaced by the same set of 
	$m$ symbols when constructing $A$.
	
	To see that the row-column intersection size in $A$ is constant 
	$\lambda_{rc}^A = \lambda_{rc}$, regardless of the order in which 
	the symbols from the different $S_i$ are entered into $A$, note that 
	in $S$, each column intersects each row in $\lambda_{rc}$ symbols, say
	the set $S_{rc}$, by definition.
	Any column in $A$ contains exactly one symbol from each set $S_k$ 
	corresponding to a symbol $s_k \in S_{rc}$, and each row in $A$ 
	contains all the symbols from a set $S_k$ corresponding to a symbol 
	$s_k \in S_{rc}$. Therefore, the row-column intersection size in $A$ 
	is still $\lambda_{rc}$.
	
	To see that the sesqui array received is \emph{proper} if $m\geq 2$, 
	note that within the $m$ first columns, the column intersection size
	is zero. However, there is a column among the columns 
	$m+1, m+2, \ldots, 2m$ such that its intersection size with 
	column $1$ is non-zero.
\end{proof}

Trying to introduce some other configuration of symbols in 
Construction~\ref{constr_product} rather than just a row with the 
symbols from $S_i$ will not yield a sesqui array, at least not 
in any straightforward way. Similarly, trying to mimic the construction 
for double or triple arrays, also does not seem to work in any straightforward 
manner. For mono arrays and AO-arrays, though, since the conditions on the 
resulting arrays are weaker, some more general constructions are possible.

\begin{construction}\label{constr_productMA}
	Let $S$ be a $(v_S, e_S, -, \lambda_{cc}^S, - : r \times c)$ mono array 
	and $T$ be a $(v_T, e_T, -, \lambda_{cc}^T, - : a \times b)$ mono array.
	
	Introduce $v_S$ sets $S_1, S_2, \ldots S_{v_S}$ of distinct symbols, 
	such that $|S_i| = v_T$	for all $i$.
	Replace each of the $v_S$ symbols $s_i$ in $S$ 
	with a $(v_T, e_T, -, -, \lambda_{rc}^T: a \times b)$ mono array isotopic 
	to $T$, on the symbol set $S_i$, producing an 
	$ra \times bc$ array $A$ on $v_Sv_T$ symbols.
\end{construction}

\begin{construction}\label{constr_productAO}
	Let $S$ be a $(v_S, e_S, -, -, \lambda_{rc}^S: r \times c)$ AO-array
	and $T$ be a $(v_T, e_T, -, -, \lambda_{rc}^T: a \times b)$ AO-array $T$.
	
	Introduce $v_S$ sets $S_1, S_2, \ldots S_{v_S}$ of distinct symbols, 
	such that $|S_i| = v_T$	for all $i$.	
	Replace each of the $v_S$ symbols $s_i$ in $S$ 
	with a $(v_T, e_T, -, -, \lambda_{rc}^T: a \times b)$ AO-array isotopic to
	$T$ on the symbol set $S_i$, producing an $ra \times bc$ array $A$
	on $v_Sv_T$ symbols.
\end{construction}

The proofs that Constructions~\ref{constr_productMA} and \ref{constr_productAO}
yield new mono arrays and AO-arrays, respectively, are similar to that of 
Theorem~\ref{thm_constr}, and are omitted. The resulting designs 
can rather easily be made to have connected row design (and column 
design)
if the number of rows (columns) in at least one of the arrays used is 
strictly greater than $1$, but we will not go into the details here.

As in Construction~\ref{thm_constr}, trivial designs $S$ and $T$ 
can give rise to non-trivial designs, for example by using Latin 
rectangles. A sufficient condition for the resulting array to be proper is that 
at least one of the arrays used in the construction is proper, but usually 
the resulting arrays do not satisfy additional intersection properties even
if both arrays are non-proper. 
In some cases, such as $4 \times 9$ arrays on $12$ symbols, proper AO-arrays
are excluded by Proposition~\ref{prop_full}, but 
Construction~\ref{constr_productAO} is still applicable, using a $4\times 3$ Latin 
rectangle and a $1 \times 3$ rectangle, and in this case yields a proper 
sesqui array. In fact, Construction~\ref{constr_productAO} can in a sense
always be used, as we prove in the following proposition.

\begin{proposition}\label{prop_AOprod}
	Construction~\ref{constr_productAO} can be used to produce AO-arrays
	for all parameter sets admissible for AO-arrays.
\end{proposition}

\begin{proof}
	Suppose $v,e,r,c$ are admissible parameters for a non-trivial AO-array, 
	so that	$v > r$, $v > c$ and $ve=rc$. Write $v = mn$ where $m|r$ and $n|c$, 
	where $m,n\geq 1$ and let $a$ and $b$ be given by $r=am$ and $c=bn$. 
	It follows that $mn = v > r = am$, so $n > a$, and similarly that $m > b$.
	
	We see then that Construction~\ref{constr_productAO} is applicable, using
	an $m \times b$ Latin rectangle on $m$ symbols, and 
	an $a \times n$ Latin rectangle on $n$ symbols, yielding an 
	$am \times bn = r \times c$ AO-array on $mn = v$ symbols.
\end{proof}

Another way of using cyclic Latin squares, that is, Latin squares that are
group tables of the cyclic group, to produce non-trivial AO-arrays 
is the following. Note that cyclic Latin squares
exist for any order.

\begin{construction}\label{constr_halflatinAO}
	Let $k$ be a positive integer and $L$ be a $2k \times 2k$ cyclic 
	Latin square $L$ on symbols	$1, 2, \ldots, 2k$ with rows and 
	columns indexed by $1, 2, \ldots, 2k$. 
	Construct an array $A$ as follows:
	
	In row $i$ of $L$, for positions $(i-1)k-(i-2), \ldots, ik-(i-1)$ taken
	modulo $2k$, if the symbol in a position is $s$, replace it with with a 
	new symbol $s'$.
\end{construction}

As an example, a $(24,12,12)$ AO-array constructed 
using Construction~\ref{constr_halflatinAO} is given in 
Figure~\ref{fig_12x12AO}.
Unfortunately, both the row design and the column design are 
disconnected.

\begin{figure}[h]\tabcolsep=4.5pt
\begin{center}
	\begin{tabular}{|cccccccccccc|}
		\hline
		1$'$	&2$'$	&3$'$	&4$'$	&5$'$	&6$'$	&7	&8	&9	&10	&11	&12\\
		2	&3	&4	&5	&6	&7$'$	&8$'$	&9$'$	&10$'$&11$'$&12$'$&1 \\
		3$'$	&4$'$	&5$'$	&6$'$	&7	&8	&9	&10 &11	&12 &1$'$	&2$'$\\
		4	&5	&6	&7$'$	&8$'$	&9$'$	&10$'$&11$'$&12$'$&1	&2	&3 \\
		5$'$	&6$'$	&7	&8	&9	&10 &11 &12 &1$'$	&2$'$	&3$'$	&4$'$\\
		6	&7$'$	&8$'$	&9$'$	&10$'$&11$'$&12$'$&1	&2	&3	&4	&5 \\
		7	&8	&9	&10	&11	&12 &1$'$	&2$'$	&3$'$	&4$'$	&5$'$	&6$'$\\
		8$'$	&9$'$	&10$'$&11$'$&12$'$&1	&2	&3	&4	&5	&6	&7$'$\\
		9	&10	&11	&12 &1$'$	&2$'$	&3$'$	&4$'$	&5$'$	&6$'$	&7	&8 \\
		10$'$	&11$'$&12$'$&1	&2	&3	&4	&5	&6	&7$'$	&8$'$	&9$'$\\
		11	&12 &1$'$	&2$'$	&3$'$	&4$'$	&5$'$	&6$'$	&7	&8	&9	&10\\
		12$'$&1	&2	&3	&4	&5	&6	&7$'$	&8$'$	&9$'$	&10$'$&11$'$\\
		\hline
	\end{tabular}
\caption{A $12 \times 12$ AO-array on $24$ symbols constructed using 
	Construction~\ref{constr_halflatinAO}.}
	\label{fig_12x12AO}
\end{center}
\end{figure}

\begin{theorem}\label{square_AO}
	For any positive integer $k$, Construction~\ref{constr_halflatinAO}
	produces a $(4k, k, -, -, k: 2k \times 2k)$ AO-array $A$.
\end{theorem}

\begin{proof}
	It is clear that $A$ is equireplicate, and each symbol 
	occurs $k$ times. It is binary since $L$ was already binary.

	With the specified pattern for replacing symbols with primed
	symbols, odd rows will have symbols $1,2,...,k$ replaced, and 
	even rows will have symbols $k+1,k+2,...,2k$ replaced.
	Odd columns will have symbols 
	$1,3,\ldots, 2\lceil\frac{k}{2}\rceil-1, 2\lfloor\frac{k}{2}\rfloor +2,
	\ldots, 2k-2,2k$ replaced 
	(i.e., the odd symbols below $k$ and the even symbols above $k+1$), 
	and even columns will have the complement of these symbols replaced.

	It is easy to check then that the intersection between any combination
	of a row and a column in $A$ is $k$. 
\end{proof}

\section{Computational methods}
\label{sec_gen}

In this section, we briefly describe the different computational methods
employed. We generated lists of admissible parameters for all types of 
arrays by a brute force search checking all divisibility conditions, 
and ran our complete enumeration
code both for all admissible parameters sets in the range up to $v=14$ and
a few larger sets of parameters. We also ran the complete enumeration code 
for some sets of inadmissible parameters to check correctness, and 
indeed found no designs on inadmissible parameter sets.

\subsection{Complete enumeration}

The basic method for our complete generation routine for an $r \times c$ 
row-column design on $v$ symbols and replication number $e$ is quite 
straightforward: 
We extend a partial array column by column, and then perform an isotopism 
reduction where only one lexicographically minimal representative for 
each isotopism class is kept. As a byproduct of this procedure, we also 
get the autotopism group sizes. 

Since our code is adapted to work primarily with constant column 
intersection sizes, we did not separately enumerate sesqui arrays and 
transposed mono arrays. Partial objects were immediately rejected if the 
requirement on constant column intersections was not met.
The AO-arrays, where neither row nor column 
intersections are required to be constant, were enumerated using the
same code, but here we did not use constant column intersection sizes 
to reject partial objects.

In the extension step, aside from constant column intersection sizes, 
we also took into account equireplication, so that no symbol was used 
too often, and we checked that the other target intersection sizes 
(row-row, row-column) were not exceeded by the partial objects.

When the full number of columns was reached,
we checked if $\lambda_{rr}$ and $\lambda_{rc}$ are constant, to classify all
received arrays as triple, (proper) double, (proper) transposed sesqui, 
(proper) mono, or AO-arrays. This procedure gives all non-isotopic row-column 
designs on a particular set of parameters.

In addition to checking isotopism, since the $4\times 4$ AO-arrays on 8 
symbols and the $6\times 6$ AO-arrays on 9 symbols can be transposed, 
for these arrays we separately computed the number of distinct equivalence
classes when also taking transposition into account.

We separately implemented an algorithm for generating our row-column designs from the 
corresponding unordered block designs, by selecting consecutive systems 
of distinct representatives. See Bailey~\cite{Ba08} for further details 
on this algorithm. It turns out that this method is not significantly 
more effective than our main method. The main limitation les in the number of partial 
objects seen in the search and the number of these become unmanageable for both methods. 
However, the results from both methods coincide for all parameter sets for which we ran both
algorithms and so the second method provides another correctness check.

The algorithms and methods used were implemented in C++ and the most 
taxing computations were run in a 
parallelized version on the Kebnekaise and Abisko supercomputers at High 
Performance Computing Centre North (HPC2N). The total run time for all 
the data in the paper was several core-years.

\subsection{Heuristic search --- Existence questions}

For a range of larger parameters where complete generation of row-column 
designs is not feasible because of the large number of designs, 
we employ a heuristic method to find examples of row-column designs of the 
different types. Because there are already a reasonably good number of known 
triple and double arrays with larger parameters, we restrict this part of 
our investigation to transposed sesqui and AO-arrays.

The heuristic search starts with a random equireplicate array 
$A$ with the selected parameters. In this array, as long as there
exists a pair of cells such that switching the symbols in these two
cells decreases the number of violations of the conditions placed
on the type of row-column design currently under consideration 
(transposed sesqui array or AO-array), we perform such a switch and
iterate. Whenever the number of violations reaches zero, the heuristic 
search algorithm
outputs the resulting array and terminates. When there are no more good 
switches to perform, the heuristic search algorithm restarts with a new array $A'$.
If the heuristic search algorithm produces an array with no violations, we check if the
design produced is proper.

\subsection{Proof of non-existence using a Boolean satisfiability model}

Our methods for complete enumeration are not well suited for ruling out 
existence of designs on larger parameter sets. We therefore also 
modeled row-column designs using Boolean satisfiability.

A rather straightforward selection of variables and {\sc Sat} clauses and 
pseudo-Boolean constraints, i.e., linear constraints over Boolean variables, 
then ensure that the desired properties are forced to hold, adapted to the 
type of design currently investigated (triple, double, sesqui, mono or AO-array).
We used the {\sc MiniSat+} solver, which is adapted for {\sc Sat} problems 
that make use of pseudo-Boolean constraints. 
See E\'en and S\"orensson~\cite{ES06} for the theory
and~\cite{minisat+} for the source code of the solver.

\subsection{Checking correctness}

Our regimen for checking the correctness of our calculations included
many standard protocols. Some correctness checks are mentioned 
elsewhere in the present paper, and here we mention some further checks.

We wrote implementations in Mathematica with which we performed an 
independent generation of the complete enumeration data for small sizes, 
in order to help verify the correctness of the C++ implementation.

Where previous enumeration results were available, we checked that 
our computational results matched those previously found. This included
in particular the number of $5\times 6$ triple arrays 
(see Phillips, Preece and Wallis~\cite{TA5x6}). 

Bagchi~\cite{Bag96} claimed that there are exactly 3 binary, equireplicate
row-column designs on 6 symbols, 3 rows and 4 columns, up to isomorphism 
(which is not explicitly defined there) and gave these three arrays. We 
have rediscovered these three arrays in our computational results, two of 
which are double arrays, and one of which is a sesqui array. We additionally 
found one more sesqui array and 3 transposed mono arrays. We double-checked 
all the additional arrays we found, and they were indeed correct and 
non-isotopic examples. 

When there was a theoretical result available, including such simple 
observations as the fact that there is a one-to-one correspondence 
between $r \times c$ and $c \times r$ proper double arrays, in the 
full enumeration we sometimes still generated both sets separately, 
and checked that the results matched. 
This includes the calculation of the autotopism groups for
double, triple and AO-arrays.

\section{Computational results and analysis}
\label{sec_compresults}

We now turn to the results and analysis of our computational work. 
With some exceptions, due to size restrictions, all the data we 
generated is available for download at~\cite{Web4}. Further details 
about the organization of the data are given there.

\subsection{The number of row-column designs for $\mathbf{v \leq 14}$}
\label{sec_totals}

We completely enumerated the number of non-isotopic proper triple, double, 
(transposed) sesqui, (transposed) mono and AO-arrays for all of the 
admissible parameter sets up to $v=10$, with the exception of 
$8\times 5$ mono arrays and the corresponding transposes which were too 
numerous, about half of the admissible parameter sets for $v=12$, and for 
some further types of arrays with $v=14$ and $v=15$. The limiting
factor was generally the number of partial objects.

The resulting numbers up to $v \leq 14$ are given in Table~\ref{tab_master}. 
An EX in the table indicates that full enumeration was too taxing, but we have 
found such arrays either by extending some of the partial objects found in an
attempt to perform complete enumeration, or in our heuristic search. Note that 
we have excluded potential rows and columns in the table where there are no 
admissible parameters for any of the designs.
In addition to $v\leq14$, we note in Table~\ref{tab_sesqui} that there are $3$ 
instances of $5 \times 6$ proper transposed sesqui arrays on $15$ symbols.

For $(v,r,c) = (10,5,6)$, there is no proper transposed sesqui array, even though 
there are all other types of designs. We have no clean and simple proof
for why this is so, which would generalize to other larger parameter sets.

When regarding transposes as ``equivalent'', which is only relevant for the
square arrays, there are $12$ different \emph{trisotopism classes} of $4\times4$ 
AO-arrays on $8$ symbols, and $\num{26632}$ different trisotopism classes of
$6\times 6$ AO-arrays on $9$ symbols (see Table~\ref{tab_AO_iso} in 
Appendix~\ref{app_auto}).

\begin{table}[p]\small
	\begin{center}
		\begin{tabular}{|l|l|c||r|r||r|r||r|r||r|} \hline
			$v$ & $e$ & $r \times c$ & MA & SA$^T$ & MA$^T$ & SA & DA & TA & AO \\ \hline
			\hline
		\multirow{2}{*}{6} 
			& \multirow{2}{*}{2}
			&   $3\times 4$ & $-$ & $-$ & 3 & 2 & 2 & 0 & $-$ \\ \cline{3-10}
			& & $4\times 3$ & 3& 2& $-$ & $-$ & 2& 0& $-$ \\ \hline\hline
		\multirow{3}{*}{8}  
			& \multirow{1}{*}{2} 
			& $4\times 4$ &  &  &  &  &  &  & 20\\ \cline{2-10}
			& \multirow{2}{*}{3}  
			&   $4\times 6$ &  &  & \num{12336} & 113 &  &  & $-$\\ \cline{3-10}
			& & $6\times 4$ & \num{12336}& 113&  &  &  &  & $-$ \\ \hline\hline
		\multirow{3}{*}{9} 
			& \multirow{2}{*}{2}
			&   $3\times 6$ &  &  & 104 & 5 &  &  & $-$ \\ \cline{3-10}
			& & $6\times 3$ & 104& 5& &  &  &  & $-$ \\ \cline{2-10}
			& \multirow{1}{*}{4}
			& $6\times 6$ &  &  &  &  &  &  & \num{53215}\\ \hline\hline
		\multirow{6}{*}{10} 
			& \multirow{2}{*}{2}
			&    $4\times 5$ & 189& 1&  &  &  &  & 45\\ \cline{3-10}
			& & $5\times 4$ &  & & 189 & 1 &  &  & 45\\ \cline{2-10}
			& \multirow{2}{*}{3}
			&   $5\times 6$ & \num{362120}& 0& \num{8364560} & 49 &\num{24663}& 7& \num{8707}\\ \cline{3-10}
			& & $6\times 5$ & \num{8364560}& 49 & \num{362120} & 0& \num{24663}& 7& \num{8707}\\ \cline{2-10}
			& \multirow{2}{*}{4}
			&   $5\times 8$ &  &  & EX & \num{1549129} &  &  & $-$ \\ \cline{3-10}
			& & $8\times 5$ & EX & \num{1549129} &  &  &  &  & $-$ \\ \cline{2-10}
			\hline\hline
		\multirow{13}{*}{12} 
			& \multirow{4}{*}{2}
			& $3\times 8$ &  &  & \num{4367} & 15 &  &  & $-$ \\ \cline{3-10}
			& & $8\times 3$ & \num{4367}& 15&  &  &  &  & $-$ \\ \cline{3-10}
			& & $4\times 6$ &  &  & \num{29695} & 20 &  &  & 312\\ \cline{3-10}
			& & $6\times 4$ & \num{29698} & 20 &  &  &  &  & 312\\ \cline{2-10}
			& \multirow{3}{*}{3}
			& $4\times 9$ & $-$ & $-$ & EX & \num{249625} & \num{2893}& 1& $-$ \\ \cline{3-10}
			& & $9\times 4$ & EX & \num{249625}& $-$ & $-$ & \num{2893}& 1& $-$ \\ \cline{3-10}
			& & $6\times 6$ &  &  &  &  &  &  & EX\\ \cline{2-10}
			& \multirow{2}{*}{4}
			& $6\times 8$ &  &  &  &  &  &  & EX\\ \cline{3-10}
			& & $8\times 6$ &  &  &  &  &  &  & EX\\ \cline{2-10}
			& \multirow{2}{*}{5}
			& $6\times 10$ &  &  & EX & EX &  &  & $-$ \\ \cline{3-10}
			& & $10\times 6$ & EX& EX&  &  &  &  & $-$ \\ \cline{2-10}
			& \multirow{2}{*}{6}
			& $8\times 9$ & EX & EX &  &  &  &  & EX\\ \cline{3-10}
			& & $9\times 8$ &  &  & EX & EX &  &  & EX\\ \cline{2-10}  
			\hline\hline
		\multirow{10}{*}{14}  
			& \multirow{2}{*}{2} 
			& $4\times 7$ &  &  &  &  &  &  & 1632\\ \cline{3-10}
			& & $7\times 4$ &  &  &  &  &  &  & 1632\\ \cline{2-10}
			& \multirow{2}{*}{3}
			& $6\times 7$ & EX& \num{44602} &  &  &  &  & EX\\ \cline{3-10}
			& & $7\times 6$ &  &  & EX & \num{44602} &  &  & EX\\ \cline{2-10}
			& \multirow{2}{*}{4}
			& $7\times 8$ & EX& EX& EX& EX& EX& EX& EX\\ \cline{3-10}
			& & $8\times 7$ & EX& EX& EX& EX& EX& EX& EX\\ \cline{2-10}			
			& \multirow{2}{*}{5}
			& $7\times 10$ &  &  &  &  &  &  & EX\\ \cline{3-10}
			& & $10\times 7$ &  &  &  &  &  &  & EX\\ \cline{2-10}
			& \multirow{2}{*}{6}			
			& $7\times 12$ &  &  & EX& EX&  &  & $-$\\ \cline{3-10}
			& & $12\times 7$ & EX& EX&  &  &  &  & $-$\\ 
		
			\hline
		\end{tabular}
	\end{center}
	\caption{All admissible parameter sets up to $v=14$, 
		with existence and the number of non-isotopic proper 
		row-column designs of different types. A dash ``$-$''
		indicates that the existence is ruled out by 
		Proposition~\ref{prop_full}, and an ``EX'' indicates that we have 
		found examples of such arrays but have no complete enumeration.
		An empty cell indicates	that the parameter set is not admissible.}
		\label{tab_master}
\end{table}

\subsection{Parameter sets with $\mathbf{v \neq r+c-1}$}

McSorley et al.~\cite{McSorley} proved that no triple array can have
$v < r+c-1$, and asked if a \emph{double} array can ever 
have $v < r+c-1$. They also gave the so far only known example of a triple 
array with $v > r+c-1$, a $TA(35,3,5,1,3:7 \times 15)$. Recall that, as 
observed in Section~\ref{sec_average}, parameters admissible 
for double arrays are automatically admissible for triple arrays.
In our search for admissible parameters sets, we also looked 
in the range excluded by the inequality $v < r+c-1$ for 
triple arrays, and there are no \emph{admissible} parameter sets for 
double/triple arrays with $v < r+c-1$ for $v \leq \num{100000}$. 
This leads us to ask the following rather number theoretical question:

\begin{question}\label{quest_small_v}
	Let $v, r, c, e$ be integers satisfying $ev=rc$, and suppose that both
	$\lambda_{rr} = \frac{c(e-1)}{r-1}$ and 
	$\lambda_{cc} = \frac{r(e-1)}{c-1}$ are integers. 
	Does it then always hold that $v \geq r+c-1$? In other words,
	are there admissible parameter sets for double/triple arrays that satisfy 
	$v < r+c-1$?
\end{question}

It can be seen in Table~\ref{tab_master} that for mono, sesqui and AO-arrays, 
there are rather small admissible parameter sets both with $v < r+c-1$ and 
$v > r+c-1$, and indeed examples of such arrays, so relaxing either the 
condition on $\lambda_{rr}$ or $\lambda_{cc}$ in Question~\ref{quest_small_v} 
gives a negative answer.
For $v < r+c-1$, the smallest such examples are $v=8$, where there are both 
proper mono arrays and transposed sesqui arrays of order $6 \times 4$, 
and $v=9$, where there are AO-arrays of order $6 \times 6$. 
For $v>r+c-1$, the smallest examples are $v=8$, 
where there are AO-arrays of order $4 \times 4$, 
and $v=9$, where there are both proper mono arrays and transposed
sesqui arrays of order $6 \times 3$. In our further search for admissible
larger parameters, we found numerous examples of both $v<r+c-1$ and 
$v>r+c-1$ for sesqui, mono and AO-arrays.

\subsection{Existence questions for $\mathbf{v \geq 15}$}
\label{sec_exist}

We present data on the existence of transposed sesqui arrays 
for $15 \leq v \leq 32$ 
in Table~\ref{tab_sesqui} and AO-arrays for $15 \leq v \leq 32$ in
Table~\ref{tab_ao}. 
Table~\ref{tab_sesqui} includes all admissible parameter sets, 
that is, parameters that were not ruled out by divisibility conditions, 
as detailed above. 

\begin{observation}\label{obs_15610}
	For $(15,6,10)$, there is no proper transposed sesqui array. This 
	was established using the SAT-model. There are, however, 
	proper AO-arrays including some found in the 
	heuristic search. This is noted with the label AO in 
	Table~\ref{tab_sesqui}. There are also triple arrays on these 
	parameters, as noted by the label T in Table~\ref{tab_sesqui}.
\end{observation}

\begin{observation}\label{obs_15910}
	For $(15,9,10)$, there is no transposed sesqui array. This was established 
	using the SAT-model. There are, however, proper AO-arrays, 
	including some found in the heuristic search.
	This is noted with the label AO in Table~\ref{tab_sesqui}.
\end{observation}

Note that in Table~\ref{tab_sesqui} there are $8$ possible parameter
sets for transposed sesqui arrays up to $n=16$. The parameters 
$(15,9,10)$ are inadmissible as observed above, but for most of the 
remaining $7$ parameter sets, the heuristic search found at least 
$20$ non-isotopic transposed sesqui arrays. In general, the heuristic search 
finds examples that do not come from the constructions.

\begin{table}\small
\begin{tabular}{|c | c || c | c || c | c || c | c |}
	\hline
	$(v,r,c)$ 	&		& $(v,r,c)$	&		& $(v,r,c)$ & 		& $(v,r,c)$ &	 \\
	\hline \hline
	$(15,10,3)$	& C$_P$,H	& $(21,14,3)$ &	C$_P$,H	& $(26,6,13)$ & C$_P$	& $(30,20,3)$ & C$_P$,H \\
	\hline
	$(15,12,5)$	& C$_P$,H	& $(21,6,7)$ &	C$_P$,H	& $(26,8,13)$ & C$_P$	& $(30,15,4)$ & C$_P$,H \\
	\hline
	$(15,5,6)$	& 3		& $(21,9,7)$ &	C$_P$,H	& $(26,12,13)$ & C$_P$	& $(30,12,5)$ & C$_P$,H \\
	\hline
	$(15,10,6)$	& H,T	& $(21,12,7)$ &	C$_P$,H	& $(26,14,13)$ & C$_P$,T	& $(30,18,5)$ & C$_P$,H \\
	\hline
	$(15,6,10)$	& AO,T	& $(21,15,7)$ &	C$_P$,H	& $(26,18,13)$ & C$_P$	& $(30,24,5)$ & C$_P$,H \\
	\hline
	$(15,9,10)$	& AO	& $(21,18,7)$ &	C$_P$,H	& $(26,20,13)$ & C$_P$	& $(30,10,6)$ & C$_P$,H \\
	\hline
	$(16,12,4)$	& C$_P$,H	& $(21,7,15)$ &	$- -$ & $(26,24,13)$ & C$_P$	& $(30,15,6)$ & C$_P$,H \\
	\hline
	$(16,14,8)$	& C$_P$,H	& $(21,14,15)$&	$- -$ & $(26,13,14)$ & ?,T & $(30,20,6)$ & C$_P$,H \\
	\hline
	$(18,12,3)$	& C$_P$,H	& $(22,10,11)$&	C$_P$	& $(27,18,3)$ & C$_P$,H		& $(30,25,6)$ & C$_P$,H,T \\
	\hline
	$(18,9,4)$	& C$_P$,H	& $(22,12,11)$&	C$_P$,T	& $(27,24,9)$ & C$_P$,H	& $(30,9,10)$ & C$_P$ \\
	\hline
	$(18,15,6)$	& C$_P$,H	& $(22,20,11)$&	C$_P$,H	& $(28,21,4)$ & C$_P$,H	& $(30,12,10)$ & C$_P$\\
	\hline
	$(18,8,9)$	& C$_P$,H	& $(22,11,12)$&	?,T	& $(28,12,7)$ & C$_P$,H	& $(30,18,10)$ & C$_P$\\
	\hline
	$(18,10,9)$	& C$_P$,T	& $(24,16,3)$&	C$_P$,H	& $(28,16,7)$ & C$_P$	& $(30,21,10)$ & C$_P$,T\\
	\hline
	$(18,16,9)$	& C$_P$,H	& $(24,12,4)$&	C$_P$,H	& $(28,24,7)$ & C$_P$,H		& $(30,27,10)$ & C$_P$,H \\
	\hline
	$(18,9,10)$	& ?,T	& $(24,18,4)$&	C$_P$,H	& $(28,7,8)$ & H		& $(30,14,15)$ & C$_P$\\
	\hline
	$(20,15,4)$	& C$_P$,H	& $(24,20,6)$&	C$_P$,H	& $(28,14,8)$ & C$_P$	& $(30,16,15)$ & C$_P$,T\\
	\hline
	$(20,8,5)$	& C$_P$,H	& $(24,21,8)$&	C$_P$,H	& $(28,21,8)$ & ?		& $(30,28,15)$ & C$_P$,H\\
	\hline
	$(20,12,5)$	& C$_P$,H	& $(24,8,9)$ &	C$_P$	& $(28,26,14)$ & C$_P$	& $(30,15,16)$ & ?,T \\
	\hline
	$(20,16,5)$	& C$_P$,H,T	& $(24,16,9)$&	C$_P$,T	& $(28,8,21)$ & $- -$	& $(30,10,21)$ & ?,T \\
	\hline
	$(20,10,6)$	& C$_P$,H	& $(24,22,12)$&	C$_P$,H	& $(28,20,21)$ & $- -$ & $(30,20,21)$ & ? \\
	\hline
	$(20,18,10)$& C$_P$,H	& $(24,9,16)$ & ?,T& 		 & 		& $(30,6,25)$ & $-$,T \\
	\hline
	$(20,5,16)$	& $-$,T& $(24,15,16)$&	?	& 			 & 		& $(30,24,25)$ & ? \\
	\hline
	$(20,15,16)$& ?	& $(25,20,5)$ &	C$_P$,H	& 			 & 		& $(32,24,4)$ & C$_P$,H \\
	\hline
	 			 &		& 			 &		& 			 & 		& $(32,28,8)$ & C$_P$,H \\
	\hline
	 			 &		& 			 &		& 			 & 		& $(32,12,16)$ & C$_P$ \\
	\hline
	 			 &		& 			 &		& 			 & 		& $(32,20,16)$ & C$_P$ \\
	\hline
	 			 &		& 			 &		& 			 & 		& $(32,30,16)$ & C$_P$ \\
	\hline
\end{tabular}
\caption{Existence of proper transposed sesqui arrays for admissible 
	parameter sets with $15 \leq v \leq 32$ symbols. 
	An ``H'' means that we found such designs in the heuristic search.
	The label ``C$_P$'' means that such an array can be constructed using 
	the transposed version of Construction~\ref{constr_product}.
	A ``$-$'' indicates that the existence
	of a proper array is excluded by Proposition~\ref{prop_full}.
	A ``$- -$'' indicates that the
	required column BIBD is non-existent, so no such proper SA$^T$ exists. 
	An ``AO'' indicates that we proved that no such SA$^T$ exists
	using the SAT-model, but that we found a proper AO-array on these parameters.
	A ``?'' means that the heuristic search did not find an example, 
	the parameter set is not covered by any construction, but we cannot rule out existence.
	An additional ``T'' indicates that there are known triple arrays.
	}
\label{tab_sesqui}
\end{table}

\begin{table}
\begin{center}
\begin{tabular}{|c | c || c | c || c | c |}
	\hline
	$(v,r,c)$ 	&		& $(v,r,c)$	&			& $(v,r,c)$ & 	 \\
	\hline \hline
	$(15,5,9)$ 	& H	& $(20,4,10)$ 	& H 	& $(24,6,8)$	& H \\
	\hline	
	$(15,10,12)$ & H	& $(20,8,10)$ 	& H	& $(24,6,12)$	& H \\
	\hline
	$(16,4,8)$ 	& H	& $(20,8,15)$ 	& H	& $(24,6,16)$	& H \\
	\hline
	$(16,6,8)$	& H	& $(20,10,10)$	& H,C$_S$	& $(24,8,12)$	& H \\
	\hline
	$(16,8,8)$	& H,C$_S$ & $(20,10,12)$ 	& 	& $(24,8,15)$	&  \\
	\hline
	$(16,8,10)$	& H	& $(20,10,14)$ 	& 	& $(24,8,18)$	& H \\
	\hline
	$(16,8,12)$	& H	& $(20,10,16)$ 	& 	& $(24,10,12)$  &  \\
	\hline
	$(16,12,12)$& 	& $(20,12,15)$ 	& 	& $(24,12,12)$	& C$_S$ \\
	\hline
	$(18,6,6)$ & H,C$_S$	& $(21,6,14)$	& H	& $(24,12,14)$ 	&  \\
	\hline
	$(18,6,9)$ & H		& $(21,9,14)$	& H	& $(24,12,16)$	&  \\
	\hline
	$(18,6,12)$ & H	& $(21,12,14)$	& 	& $(24,12,18)$	&  \\
	\hline
	$(18,9,12)$ & H	& $(21,14,18)$	& 	& $(24,12,20)$	&  \\
	\hline
	$(18,9,14)$ & H	& $(22,4,11)$	& H	& $(24,16,18)$	&  \\
	\hline
	$(18,12,12)$& C$_S$	& $(22,6,11)$	& H	& $(24,16,21)$	&  \\
	\hline
	$(18,12,15)$& 	& $(22,8,11)$	& H	& $(24,18,20)$	&  \\
	\hline
		& & $(22,11,14)$	& 	&	&  \\
	\hline
		& & $(22,11,16)$	& 	&	&  \\
	\hline
				& 		& $(22,11,18)$	& 	&	&  \\
	\hline
		
\end{tabular}
\caption{Existence of proper AO-arrays with $r \leq c$ on $15 \leq v \leq 24$ 
	symbols, for parameter sets only admissible for AO-arrays. The label ``H'' 
	means that we found such designs in the heuristic search. 
	The label ``C$_S$'' means that such a design can be constructed using 
	Construction~\ref{constr_halflatinAO}. Note that
	Construction~\ref{constr_productAO} is also applicable for any admissible
	parameter set.}
\label{tab_ao}
\end{center}
\end{table}

\subsection{Autotopism group orders}
\label{sec_autotopism}

Data on the distribution of autotopism group orders are collected in 
Appendix~\ref{app_auto}, in 
Tables~\ref{tab_tada_iso} (for double and triple arrays), 
\ref{tab_mono_iso} (for mono arrays), 
\ref{tab_sesqui_iso} (for transposed sesqui arrays), and 
\ref{tab_AO_iso} (for AO-arrays).
We see that generally the most common autotopism group order is $1$,
but that there are also some very symmetric objects.

The designs constructed using 
Constructions~\ref{constr_product}, \ref{constr_productMA} and 
\ref{constr_productAO} retain the autotopisms inherited from the component
arrays. For example, if the AO-array $A$ is constructed from AO-arrays $B$ 
and $C$ using the canonical ordering, and denoting by $Aut(X)$ the 
autotopism group of the array $X$, the direct product of $Aut(B)$ and 
$Aut(C)$ is a subgroup of $Aut(A)$.

\section{Relation to other design types}
\label{sec_other_rel}

There is a rather rich flora of other closely related design types with 
ordered blocks. Here we briefly investigate the relation between, on the 
one hand, the row-column designs in the present paper, and on the other 
hand, Youden rectangles and binary Pseudo Youden designs.

\subsection{Relation to Youden rectangles}
\label{sec_youden_rel}

An $(n,k,\lambda)$ Youden rectangle is a binary $k \times n$ array on $n$
symbols, where each pair of columns have exactly $\lambda= k(k-1)/(n-1)$
symbols in common, or, equivalently, where each pair of symbols appears
together in exactly $\lambda$ columns.
Some of the arrays treated in the present paper can be constructed from
a Youden rectangle $Y$ with suitable parameters, by picking a column $C$ 
in $Y$ with symbol set $S$, removing all the symbols in $S$ from $Y$, 
removing column $C$, and then exchanging the roles of columns and symbols. 
Regarding parameters, an $(n,k,\lambda)$ Youden rectangle gives rise to 
an array on $v = n-1$ symbols, $r=k$ rows, $c=n-k$ columns in this way. 
It was proved by Nilson and \"Ohman in~\cite{NO15} that the resulting 
array using this
transformation is always binary, equireplicate with replication number 
$e=k-\lambda$, and has constant column/column intersections of size 
$\lambda_{cc}=\lambda$, that is, it is a mono array. It was also proven there
that any triple array with $\lambda_{cc}=2$ can be produced from a
suitable Youden rectangle in this way.

In~\cite{YoudenEnum} the present authors completely enumerated Youden rectangles 
for small parameters, and investigated which rectangles give rise to triple arrays, 
double arrays and transposed sesqui arrays using the transformation 
described above. These previous results (numerators) are summarized in 
Table~\ref{tab_youden}, together with the total number (denominators) of 
designs of the appropriate type that 
we found in the present enumeration. We see that for some but not all 
parameters, we receive all triple and proper double arrays in this way,
but not all proper transposed sesqui arrays for any of the parameters in 
this range. It was asked by Nilson and \"Ohman in~\cite{NO15} if any Youden 
rectangle yields a double array by this transformation at least for 
\emph{some} column. The present authors answered this in the negative by 
providing counterexamples in~\cite{YoudenEnum}, and we can now supplement 
this by observing that not all double arrays can be constructed in this way.

\begin{table}[ht]
\begin{center}
\begin{tabular}{|c|c||r|r|r|}
	\hline 
	YR$(n,k,\lambda)$ & Array params. & TA & DA & SA$^T$ \\ 
	\hline 	\hline
	(7,3,1)  & $(6:3\times 4)$ & 0/0 & 1/2 & 0/0 \\ 
	\hline 	
	(7,4,2)  & $(6:4\times 3)$ & 0/0 & 2/2 & 1/2 \\ 
	\hline 	
	(11,5,2) & $(10:5\times 6)$ & 7/7 & \num{17642}/\num{24663} & 0/0 \\ 
	\hline 	
	(11,6,3) & $(10:6\times 5)$ & 7/7 & \num{24663}/\num{24663} & 34/49 \\ 
	\hline 	
	(13,4,1) & $(12:4\times 9)$ & 0/1 & 192/\num{2893} & 0/0 \\ 
	\hline 	
\end{tabular} 
\end{center}
\caption{The proportion of proper row-column designs that can be constructed 
	from Youden rectangles by removing a column, all the symbols therein, and
	then exchanging the roles of symbols and columns.}
\label{tab_youden}
\end{table}

In the course of our computations for Table~\ref{tab_youden}, we 
rediscovered (see our previous paper~\cite{YoudenEnum}) some highly symmetric 
Youden rectangles that give rise to a triple array for any choice of removed 
column. We also noted in~\cite{YoudenEnum} that there are examples
of Youden rectangles that do not yield double arrays for any choice of 
removed column. Between these extremes, we also found Youden rectangles 
that give double arrays when removing some column.

\subsection{Relation to binary pseudo Youden designs}
\label{sec_pseudo_rel}

A class of row-column designs called \emph{Pseudo Youden designs} (PYD) 
was introduced by Cheng~\cite{Ch81}. Adding the assumption of binarity,
a \emph{binary} PYD is an $r \times r$ binary and equireplicate row-column 
design on $v$ symbols, denoted by $PYD(v:r\times r)$,
such that when taking the rows and columns as $2r$ blocks, they form a 
BIBD, that is, any pair of symbols occurs in a constant number of blocks.
Since block sizes in a BIBD are constant, the row-column design has
to be square, that is, $r=c$, in order to be a PYD. The replication 
number of symbols in the row-column design, $e=\frac{r^2}{v}$, has to 
be an integer, and the symbol replication in the corresponding BIBD 
will be twice this number, since each entry is counted once among the 
rows, and once among the columns. Additionally, the pair replication 
number in the BIBD $\lambda = 2e(r-1)/(v-1)$ must be integer, by 
standard BIBD theory, where we note that $e$ and $r$ are not the 
parameter names usually used for BIBD:s. 

Cheng~\cite{Ch81b} constructed an infinite family of binary 
$PYD(s^2:s(s+1)/2 \times s(s+1)/2)$ for $s\equiv 3 \pmod 4$ where
$s$ is a prime or prime power, but we have found no general 
characterization of admissible parameters for PYD:s in the literature. 
We therefore prove the following proposition, which covers all the 
parameter sets in Cheng's construction.

\begin{proposition}\label{prop_PYD_param}
	Let $s_i$ be the $i$:th odd number, and $t_i$ be the $i$:th
	even triangular number. Then for $i\geq2$, $v=s_i^2$ symbols and  
	$r= t_i$ is a set of admissible parameters 	for an $r\times r$  PYD.
\end{proposition}

\begin{proof}
	Let $e$ be the replication number in the row-column designs (to be 
	calculated), and consequently $2e$ be the replication number in the 
	BIBD corresponding to these PYD parameters.
	
	The standard divisibility conditions for the BIBD give that we must 
	satisfy $2ve=br$, where $v=s_i^2$ is the number of symbols, 
	$b=2t_i$ is the number of blocks in the BIBD, and $r=t_i$ is the size 
	of the blocks. Additionally, for the pair replication number $\lambda$ 
	in the BIBD we must have $\lambda(v-1) = 2e(r-1)$.
	
	The $i$:th odd square can be written as $s_i^2 = (2i-1)^2$, and the 
	$i$:th even triangular number can be written as 
	$(2i-1)(2i-1-(-1)^{(i-1)})/2$. The replication number in the 
	row-column design will then be 
	$e = t_i^2/s_i^2 = (2i-1-(-1)^{(i-1)})^2/4$, 
	which is clearly integer since the numerator is a square of an even 
	number and non-zero if $i\geq2$, and the corresponding replication 
	number of the BIBD is $2e = (2i-1-(-1)^{(i-1)})^2/2$, which is of 
	course also integer and non-zero.
	
	The pair replication number in the BIBD will then be 
	\[
		\lambda 
		= \frac{2e(r-1)}{v-1} 
		= \frac{(2i-1-(-1)^{(i-1)})^2(((2i-1)(2i-1-(-1)^{(i-1)}))/2-1)}{2((2i-1)^2-1)}.
	\]
	For convenience, we treat odd and even $i$ separately. In the even case,
	by some basic algebra, the expression for $\lambda$ simplifies to 
	$i(2i+1)/2$, which
	is clearly integer. In the odd case, we instead get $(i-1)(2i-3)/2$,
	which again is integer and non-zero for $i\geq2$.
	
	Since all divisibility criteria for BIBD are thus met, we conclude
	that the parameters are admissible.	 
\end{proof}

We note that the replication numbers of the row-column design parameters
are consecutive squares of even numbers, but they only go up for every 
second increase of $i$. We also ran a brute-force search for other 
admissible PYD parameters, and up to $v=367^2$ there are no such parameters
for anything other than $v$ being an odd square. However, for some values 
of $i$, the first being $i=17$, we found other possible dimensions for 
the PYD. For $i=17$, the main series of parameters given in 
Proposition~\ref{prop_PYD_param} predicts $PYD(289:136 \times 136)$,
but additionally, $PYD(289:204 \times 204)$ is also admissible. For $i = 99$
we found the first instance of three admissible parameter sets, namely
$PYD(99: 3465 \times 3465)$, $PYD(99: 4950 \times 4950)$, and
$PYD(99: 6930 \times 6930)$, where $r=4950$ is what 
Proposition~\ref{prop_PYD_param} predicts.

Among the parameter sets admissible for PYD:s, some are admissible for 
AO-arrays, but as we now prove, none of the other types of row-column 
designs we treat here can be square in the non-trivial range for $v$. 
Note that for $r=v$, we get a Latin square, which trivially satisfies 
all the intersection conditions.

\begin{proposition}\label{prop:square}
	Suppose $r=c$, $r < v < r^2$ and the parameters $(v,r,r)$ are 
	admissible for a row-column design $A$. Then $A$ is a proper 
	AO-array.
\end{proposition}

\begin{proof}
	As observed in Section~\ref{sec_average}, if a parameter set is
	admissible for some row-column design, then it is admissible for 
	an AO-array, so $\lambda_{rc}$ is integer.
	
	If we additionally want, say, $\overline{\lambda_{rr}}=\frac{r(e-1)}{r-1}$ 
	to be integer, $r-1$ clearly has to divide $e-1$. If $e=1$, we get an array
	where each symbol appears only once, which trivially is a triple array, but 
	since $v<r^2$, it holds that $e>1$, and it follows that for some
	integer $a\geq1$, 
	$e-1 = a(r-1)$. Substituting $e=\frac{r^2}{v}$ and solving for $a$,
	it follows that $1 \leq a = (r^2-v)/(vr-v)$, so $r\geq v$, a contradiction
	to the assumption that $r < v$.
	
	Since $\overline{\lambda_{cc}}=\overline{\lambda_{rr}}=\frac{r(e-1)}{r-1}$
	in a square design, the same argument holds if we instead assume that
	$\overline{\lambda_{cc}}$ is integer.
	It follows that $A$ is a proper AO-array.
\end{proof}

For square AO-arrays, we have fully enumerated $4 \times 4$ AO-arrays on $8$ 
symbols and $6 \times 6$ AO-arrays on $9$ symbols. By 
Proposition~\ref{prop_AOprod}, AO-arrays exist for any admissible parameter set,
and we additionally constructed $12 \times 12$ AO-arrays on $16, 18$ and $24$ 
symbols, respectively, using Construction~\ref{constr_halflatinAO}.
In our heuristic search, we have also found additional examples of AO-arrays 
of sizes $6 \times 6$ on 12 or 18 symbols, $8 \times 8$ on 16 symbols, 
and $10 \times 10$ on 20 symbols (see Table~\ref{tab_ao}). 

Most of these parameters sets are ruled out for PYD:s, since the 
corresponding BIBD does not exist. In particular, an $r \times r$ 
AO-array on $v$ symbols and replication number $e$ would 
give rise to a $(v,2r,2e,r,\frac{2e(r-1)}{v-1})$-BIBD, or 
$(v,r,\frac{2e(r-1)}{v-1})$-BIBD for short. Among our examples, this only 
leaves $6 \times 6$ AO-arrays on $9$ symbols as candidates for being PYD:s.

McSorley and Philips~\cite{PYD6x6} enumerated and analyzed in great 
detail all $6 \times 6$ PYD:s on $9$ symbols. They found $696$ non-isotopic
PYD:s, which reduced to $348$ species when allowing transposes, since
no PYD was found to be isotopic to its own transpose.

Among our data, $157$ of the $\num{53125}$ non-isotopic $6 \times 6$ AO-arrays on $9$ 
symbols were found to also satisfy the PYD condition, one of which is 
given in Figure~\ref{fig_AO_PYD}, and we found that this reduces to $153$
AO-arrays when considering transposes to be equivalent. We note that this shows
that not all PYD:s are AO-arrays, and that not all AO-arrays on suitable 
parameters are PYD:s. Construction~\ref{constr_productAO} does not in general
yield PYD:s. In~\cite{PYD6x6},
several further properties of $6 \times 6$ PYD:s are investigated, but 
such further investigations of the AO-arrays are beyond the scope of 
the present paper.

\begin{figure}\tabcolsep=4.5pt
\begin{center}\begin{tabular}{|cccccc|}
	\hline
	0 &1 &2 &3 &4 &5 \\
	1 &2 &0 &4 &6 &7 \\
	2 &3 &8 &6 &5 &0 \\
	3 &4 &7 &5 &8 &1 \\
	7 &8 &5 &2 &1 &6 \\
	8 &6 &4 &7 &0 &3 \\
	\hline
	\end{tabular}
\end{center}
	\caption{A $6 \times 6$ AO-array on $9$ symbols that is also a PYD.}
	\label{fig_AO_PYD}
\end{figure}

\section{Open questions}
\label{sec_concl}

It can be observed in Table~\ref{tab_master}
that there are only a few admissible parameter sets for which the 
corresponding design does not exist. One such set, which has already 
been discussed in the literature is $(6,3,4)$, for which there exists
no triple array.
The other parameter set in Table~\ref{tab_master} is $(10,5,6)$, 
where there are no proper transposed sesqui arrays, and $(10,6,5)$ 
where there are no proper sesqui arrays. We have also observed above
in Observations~\ref{obs_15610} and~\ref{obs_15910} that there are no 
proper transposed sesqui arrays on parameters $(15,6,10)$ and 
$(15,9,10)$. In all these cases, the relevant component designs
exist.
This leads us to ask the following question.

\begin{question}
	Is there some `simple' necessary condition for existence 
	that rules out proper transposed sesqui arrays on 
	parameters $(10,5,6)$, $(15,6,10)$ or $(15,9,10)$? 
\end{question}

We would also like to restate Question~\ref{quest_small_v}.

\begin{question}
	Are there admissible parameter sets for double/triple arrays that satisfy 
	$v < r+c-1$?
\end{question}

\section*{Acknowledgments}

We would like to thank Alexey Gordeev for independently verifying the 
computational results in Tables~\ref{tab_tada_iso} 
through~\ref{tab_AO_iso} in Appendix~\ref{app_auto}.

The computational work was performed on resources provided by the Swedish 
National Infrastructure for Computing (SNIC) at 
High Performance Computing Center North (HPC2N). 
This work was supported by the Swedish strategic research programme eSSENCE.    
This work was supported by The Swedish Research Council grant \num{2014}-\num{4897}. 


\bibliographystyle{plain}


\appendix
\newpage
\section{Autotopism group orders}\label{app_auto}

\begin{table}[H]
	\begin{center}
		\begin{tabular}{|c|r||r|r|r||r|r|}
			\multicolumn{2}{r||}{} & \multicolumn{3}{c||}{Double arrays} 
				& \multicolumn{2}{c}{Triple arrays} \\
				\hline
			\multicolumn{2}{|r||}{$v$} 
				& $6$ & $10$ & $12$ & $10$ & $12$ \\ 
				\hline 
			\multicolumn{2}{|r||}{$r \times c$} 
				& $3 \times 4$	& $5 \times 6$ & $4 \times 9$ 
				& $5 \times 6$ & $4 \times 9$ \\ 
				\hline
			\multicolumn{2}{|r||}{Total \#} 
				& 2 & \num{24663} & 2893 & 7 & 1 \\ 
			\hline\hline
			$|\mbox{Aut}|$ 
			& 1 &   & \num{24146} & 2867 &   &   \\ 
			\cline{2-7}
			& 2 & 1 & 398   &      &   &   \\ 
			\cline{2-7}
			& 3 & 1 & 89    & 24   & 2 & 1 \\ 
			\cline{2-7}
			& 4 &   & 13    &      & 1 &   \\ 
			\cline{2-7}
			& 5 &   & 5     &      &   &   \\ 
			\cline{2-7}
			& 6 &   & 8     &      & 1 &   \\ 
			\cline{2-7}
			& 9 &   &       & 2    &   &   \\ 
			\cline{2-7}
			& 10 &  & 1     &      &   &   \\ 
			\cline{2-7}
			& 12 &  & 3     &      & 2 &   \\ 
			\cline{2-7}
			& 60 &  &       &      & 1 &   \\ 
			\hline 
		\end{tabular} 
	\end{center}
	\caption{The number of proper double arrays and triple arrays sorted 
		by autotopism group order.}
	\label{tab_tada_iso}
\end{table}

\begin{table}[H]
	\begin{center}
		\begin{tabular}{|c|r||r|r|r|r|r|r|r|r|}
			\hline
			\multicolumn{2}{|r||}{$v$} & $6$ & $8$ & $9$ 
				& \multicolumn{3}{c|}{$10$} & \multicolumn{2}{c|}{$12$} \\ 
				\hline 
			\multicolumn{2}{|r||}{$r \times c$} 
				& $4 \times 3$	& $6 \times 4$ & $6 \times 3$ 
				& $4 \times 5$ & $5 \times 6$ & $6 \times 5$ 
				& $6 \times 4$ & $8 \times 3$ \\ 
				\hline
			\multicolumn{2}{|r||}{Total \#} 
				& 3	& \num{12336} & 104 & 189 & \num{362120} & \num{8364560} & \num{29695} & 4367 \\ 
			\hline\hline
			$|\mbox{Aut}|$ 
			& 1 &   & \num{11643} & 65  & 140 & \num{360485} & \num{8357136} & \num{28007} & 3970 \\ 
			\cline{2-10}
			& 2 & 2 & 598   & 31  & 40  & 1610   & 6890    & 1492  & 338  \\ 
			\cline{2-10}
			& 3 &   & 19    & 2   &     & 14     & 423     & 28    & 9    \\ 
			\cline{2-10}
			& 4 & 1 & 58    & 1   & 7   &        & 86      & 125   & 29   \\ 
			\cline{2-10}
			& 5 &   &       &     &     & 5      & 4       &       &      \\ 
			\cline{2-10}
			& 6 &   & 7     & 5   &     & 4      & 18      & 19    & 14   \\ 
			\cline{2-10}
			& 8 &   & 10    &     & 2   &        &         & 18    & 6    \\ 
			\cline{2-10}
			& 10 &   &      &     &     & 2      & 3       &       &      \\ 
			\cline{2-10}
			& 12 &   & 1    &     &     &        &         & 3     &      \\ 
			\cline{2-10}
			& 16 &   &      &     &     &        &         & 3     & 1    \\ 
			\cline{2-10}
			\hline 
		\end{tabular} 
	\end{center}
	\caption{The number of proper mono arrays sorted by autotopism group order.}
	\label{tab_mono_iso}
\end{table}


\begin{table}\small
	\begin{center}
		\begin{tabular}{|c|r||r|r|r|r|r|r|r|r|r|r|r|}
			\hline
			\multicolumn{2}{|r||}{$v$} & $6$ & $8$ & $9$ 
				& \multicolumn{3}{c|}{$10$} & \multicolumn{3}{c|}{$12$} 
				& $14$ & $15$ \\ 
				\hline 
			\multicolumn{2}{|r||}{$r \times c$} 
				& $4 \times 3$ & $6 \times 4$ & $6 \times 3$ 
				& $4 \times 5$ & $6 \times 5$ & $8 \times 5$
				& $6 \times 4$ & $8 \times 3$ & $9 \times 4$ &
				$6 \times 7$ & $5 \times 6$ \\ 
				\hline
			\multicolumn{2}{|r||}{Total \#} 
				& 2	& 113 & 5  & 1 & 49 & \num{1549129} & 20 & 15 & \num{249625} & \num{44602} & $3$ \\ 
			\hline\hline
			$|\mbox{Aut}|$ 
			& 1 &   & 40  &    &   & 31 & \num{1537034} &   & 1 & \num{243241} & \num{40617} & \\ 
			\cline{2-13}
			& 2 &   & 35  & 1  &   & 9  & \num{11617} & 3  & 4 & 5660 & 3887 & 1 \\ 
			\cline{2-13}
			& 3 &   & 1   &    &   & 3  & 148 &   &   & 86 & 41  & \\ 
			\cline{2-13}
			& 4 & 1 & 17  & 1  &   & 1  & 214 & 4 &   & 484 &  &  \\ 
			\cline{2-13}
			& 5 &   &     &    &   &    & 4 &   &   &  &  & \\ 
			\cline{2-13}
			& 6 &   & 1   & 1  &   & 4  & 84 & 2  & 2 & 52 & 55 & 2 \\ 
			\cline{2-13}
			& 8 &   & 11  &    &   &    & 23 & 6  & 4 & 63 &  & \\ 
			\cline{2-13}
			& 9 &   &    &    &   &    &   &   &   & 1 &  & \\ 
			\cline{2-13}
			& 12 & 1 &    &    &   &    &   &   &   & 15 &  & \\ 
			\cline{2-13}
			& 16 &   & 4  &    &   &    &   & 2 & 1 & 9 &  & \\ 
			\cline{2-13}
			& 18 &   &    & 1  &   &    &   &   & 1 & 2 & & \\ 
			\cline{2-13}
			& 20 &   &    &    & 1 & 1  &   &   &   &  & & \\ 
			\cline{2-13}
			& 24 &   & 2  &    &   &    & 4 & 2 & 1  & 9 & & \\ 
			\cline{2-13}
			& 36 &   &    & 1  &   &    &   &   &    &  & & \\ 
			\cline{2-13}
			& 40 &   &    &    &   &    & 1 &   &    &  & & \\ 
			\cline{2-13}
			& 42 &   &    &    &   &    &   &   &    &  & 2 & \\ 
			\cline{2-13}
			& 48 &   & 2  &    &   &    &   &   &    & 2 & & \\ 
			\cline{2-13}
			& 96 &   &    &    &   &    &   & 1 &    &  & & \\ 
			\cline{2-13}
			& 144 &   &    &   &   &    &   &   & 1  & 1 & & \\
			\hline 
		\end{tabular} 
	\end{center}
	\caption{The number of transposed proper sesqui arrays sorted by 
		autotopism group order.}
	\label{tab_sesqui_iso}
\end{table}

\begin{table}
	\begin{center}
		\begin{tabular}{|c|r||r|r|r|r|r|r||r|r|}
			\multicolumn{2}{r||}{} & \multicolumn{6}{c||}{Autotopism} 
				& \multicolumn{2}{c}{Autotrisotopism} \\
			\hline
			\multicolumn{2}{|r||}{$v$} 
				& $8$ & $9$ & \multicolumn{2}{c|}{$10$} & $12$ & $14$ & 8 & 9 \\ 
				\hline 
			\multicolumn{2}{|r||}{$r \times c$} 
				& $4 \times 4$ & $6 \times 6$ & $4 \times 5$ & $5 \times 6$ 
				& $4 \times 6$ & $4 \times 7$ & $4 \times 4$ & $6 \times 6$\\ 
				\hline
			\multicolumn{2}{|r||}{Total \#} 
			   & 20 & \num{53215} & 45 & 8707 & 312 & 1632 & 12 & \num{26632} \\ 
			\hline\hline
			$|\mbox{Aut}|$ 
			& 1 &   & \num{49280} & 3  & 7534 & 38  & 641 & & \num{24634} \\ 
			\cline{2-10}
			& 2 & 1 & 3488  & 15 & 1042  & 105 & 593 & 1 & 1746 \\ 
			\cline{2-10}
			& 3 &   & 105    &   & 8    &     &     &  & 52 \\ 
			\cline{2-10}
			& 4 & 4 & 238    & 16 & 113   & 76  & 232 & 4 & 133 \\ 
			\cline{2-10}
			& 5 &   &       &   & 1    &     &     & &  \\ 
			\cline{2-10}
			& 6 &   & 70    &   & 2    & 2   & 9   & & 36 \\ 
			\cline{2-10}
			& 8 & 7 &       & 4 &      & 40  & 60  & 3 & 12 \\ 
			\cline{2-10}
			& 9 &  & 2      &   &      &     &     & & 1 \\ 
			\cline{2-10}
			& 10 &   &      &   & 2    &     & 1   & &  \\ 
			\cline{2-10}
			& 12 &   & 18    & 3 & 4    & 12  & 36  & & 8 \\ 
			\cline{2-10}
			& 14 &   &      &   &      &     & 1   & &  \\ 
			\cline{2-10}
			& 16 & 5 &      &   &      & 15  & 13  & 2 &  \\ 
			\cline{2-10}
			& 18 &  & 6     &   &      &   &   &  &  3 \\ 
			\cline{2-10}
			& 20 &   &      & 1 & 1    &     & 5   & &  \\ 
			\cline{2-10}
			& 24 &   &      & 3 &      & 5   & 18  & & 2 \\ 
			\cline{2-10}
			& 28 &   &      &   &      &     & 1   & &  \\ 
			\cline{2-10}
			& 32 & 2 &      &   &      & 8   & 3   & 1 & \\ 
			\cline{2-10}
			& 36 &   & 8     &   &      & 1   &     & & 3 \\ 
			\cline{2-10}
			& 40 &   &      &   &      &     & 4   & &  \\ 
			\cline{2-10}
			& 48 &   &      &   &      & 2   & 12  & &  \\ 
			\cline{2-10}
			& 64 & 1 &      &   &      & 3   &     & 1 &  \\ 
			\cline{2-10}
			& 72 &   &      &   &      & 4   &     & & 2 \\ 
			\cline{2-10}
			& 96 &   &      &   &      &     & 3   & &  \\ 
			\cline{2-10}
			& 384&   &      &   &      & 1   &     & &  \\ 
			\cline{2-10}
			\hline 
		\end{tabular} 
	\end{center}
	\caption{The number of proper AO-arrays sorted by autotopism group order,
		or autotrisotopism group order, respectively.}
	\label{tab_AO_iso}
\end{table}

\end{document}